\documentclass[final,3p]{elsarticle}
\usepackage{lineno,hyperref}
\usepackage{amsmath}
\usepackage{amssymb}
\usepackage[linesnumbered,ruled,vlined]{algorithm2e}
\usepackage{caption}
\usepackage{mathtools}
\usepackage{changes}
\usepackage{multirow}
\usepackage{verbatim}
\usepackage{mathrsfs}
\usepackage{graphicx}
\usepackage{subcaption}
\usepackage{amsmath,amsthm,bm,mathrsfs}
\theoremstyle{plain}
\newtheorem{theorem}{Theorem}[section]

\newtheorem{proposition}[theorem]{Proposition}

\theoremstyle{definition}

\theoremstyle{remark}
\newtheorem{remark}{Remark}

\modulolinenumbers[5]

\journal{ArXiv.org}

\begin{document}

\begin{frontmatter}

\title{Numerical Solution of Large-Scale Low-Rank Coupled Generalized Sylvester Equations}

\author[uz]{Umair~Zulfiqar\corref{mycorrespondingauthor}}
\cortext[mycorrespondingauthor]{Corresponding author}
\ead{umair@yangtzeu.edu.cn}
\address[uz]{School of Electronic Information and Electrical Engineering, Yangtze University, Jingzhou, Hubei, 434023, China}

\begin{abstract}
This paper develops projection methods for computing low-rank solutions to large-scale coupled generalized Sylvester equations. The dominant subspaces of these solutions are shown to be spanned by specific resolvent directions, motivating a Petrov--Galerkin projection framework. Conditions for the unique solvability of the projected coupled generalized Sylvester equations are also discussed. Based on these dominant subspaces, two algorithms are proposed: an extended Krylov-based projection algorithm and a rational Krylov-based projection algorithm. The former uses fixed shifts, whereas the latter uses adaptive shifts. Numerical experiments on large-scale benchmark problems demonstrate that the proposed methods achieve high accuracy and fast convergence.
\end{abstract}


\end{frontmatter}

\section{Introduction}
Coupled matrix equations arise in various areas of applied mathematics and control theory. An important motivation comes from descriptor systems, where converting differential-algebraic equations (DAEs) into equivalent ordinary differential equation (ODE) state-space representations requires solving coupled generalized Sylvester equations at the algebraic level \cite{gerdin2007parameter,shahzad2011efficient}. This transformation is essential for controller design and numerical simulation of DAE-based models.

Consider the coupled generalized Sylvester equations
\begin{align}
A_1 X + Y A_2  &= - B_1 C_1, \nonumber\\
E_1 X + Y E_2  &= - B_2 C_2, \label{eq:coupled}
\end{align}
where $X, Y \in \mathbb{R}^{n_1 \times n_2}$,
$A_1, E_1 \in \mathbb{R}^{n_1 \times n_1}$,
$A_2, E_2 \in \mathbb{R}^{n_2 \times n_2}$,
$B_1, B_2 \in \mathbb{R}^{n_1 \times m}$, and
$C_1, C_2 \in \mathbb{R}^{m \times n_2}$, with
$m \ll \min\{n_1, n_2\}$.

The following assumptions hold throughout the paper:
\begin{enumerate}
    \item[(A1)] The matrices $E_1$ and $E_2$ are nonsingular.
    \item[(A2)] The spectra of $E_1^{-1}A_1$ and $E_2^{-1}A_2$ are disjoint:
    \[
    \Lambda(E_1^{-1}A_1) \cap \Lambda(E_2^{-1}A_2) = \emptyset.
    \]
    This condition ensures that the coupled operator is bijective, ensuring that \eqref{eq:coupled} has a unique solution for any right-hand side \cite{wimmer1994consistency,kagstrom1989generalized,kaagstrom2000extracting,jones2009modeling}.
    \item[(A3)] The solutions $X$ and $Y$ are numerically low-rank. Because the right-hand sides possess a low-rank structure, the singular values of $X$ and $Y$ decay rapidly, allowing accurate low-rank approximations.
\end{enumerate}
\section{Main Work}
This section first presents a general projection framework without prescribing the projection matrices and discusses conditions for the unique solvability of the projected coupled generalized Sylvester equations. An analysis of the structural properties of low-rank solutions then guides the choice of projection matrices, leading to two projection algorithms for computing low-rank solutions to large-scale coupled generalized Sylvester equations.
\subsection{Petrov--Galerkin Projection Framework}
Projection methods are widely used to compute low-rank solutions to matrix equations \cite{simoncini2016computational}. For the coupled generalized Sylvester equations \eqref{eq:coupled}, consider approximations of the form
\begin{equation}\label{eq:ansatz}
\tilde{X} = \hat{V}_x \hat{X} \hat{W}_x^\top, \qquad
\tilde{Y} = \hat{V}_y \hat{Y} \hat{W}_y^\top,
\end{equation}
where $\hat{V}_x, \hat{V}_y \in \mathbb{R}^{n_1 \times r}$ and $\hat{W}_x, \hat{W}_y \in \mathbb{R}^{n_2 \times r}$ have orthonormal columns, and $\hat{X}, \hat{Y} \in \mathbb{R}^{r \times r}$. With the residuals
\begin{align*}
R_1(\tilde{X}, \tilde{Y}) &= A_1 \tilde{X} + \tilde{Y} A_2 + B_1 C_1, \\
R_2(\tilde{X}, \tilde{Y}) &= E_1 \tilde{X} + \tilde{Y} E_2 + B_2 C_2,
\end{align*}
the Petrov--Galerkin conditions
\begin{equation}\label{eq:galerkin}
\hat{V}_y^\top R_1(\tilde{X}, \tilde{Y}) \hat{W}_x = 0, \qquad
\hat{V}_y^\top R_2(\tilde{X}, \tilde{Y}) \hat{W}_x = 0
\end{equation}
yield the projected coupled generalized Sylvester equations
\begin{align}
\hat{A}_1 \hat{X} + \hat{Y} \hat{A}_2  &= - \hat{B}_1 \hat{C}_1, \nonumber\\
\hat{E}_1 \hat{X} + \hat{Y} \hat{E}_2  &= - \hat{B}_2 \hat{C}_2, \label{eq:projected}
\end{align}
with
\begin{align*}
\hat{A}_1 &= \hat{V}_y^\top A_1 \hat{V}_x, \quad
\hat{E}_1 = \hat{V}_y^\top E_1 \hat{V}_x, \quad
\hat{B}_1 = \hat{V}_y^\top B_1, \quad
\hat{B}_2 = \hat{V}_y^\top B_2, \\
\hat{A}_2 &= \hat{W}_y^\top A_2 \hat{W}_x, \quad
\hat{E}_2 = \hat{W}_y^\top E_2 \hat{W}_x, \quad
\hat{C}_1 = C_1 \hat{W}_x, \quad
\hat{C}_2 = C_2 \hat{W}_x.
\end{align*}
Under the standing assumptions, a specific choice of the test bases $\hat{V}_y$ and $\hat{W}_x$ ensures the well-posedness of \eqref{eq:projected}.

\begin{theorem}\label{thm:projected}
Let $\hat{V}_x \in \mathbb{R}^{n_1 \times r}$ and $\hat{W}_y \in \mathbb{R}^{n_2 \times r}$ be arbitrary matrices with orthonormal columns, and define
\begin{equation}\label{eq:test_bases}
\hat{V}_y = \operatorname{orth}(E_1 \hat{V}_x), \qquad
\hat{W}_x = \operatorname{orth}(E_2^\top \hat{W}_y).
\end{equation}
Assume that $\hat{E}_1 = \hat{V}_y^\top E_1 \hat{V}_x$ and $\hat{E}_2 = \hat{W}_y^\top E_2 \hat{W}_x$ are nonsingular. If
\begin{equation}\label{eq:fov_condition}
\mathcal{F}(A_1 E_1^{-1}) \cap \mathcal{F}(E_2^{-1} A_2) = \emptyset,
\end{equation}
then the projected coupled generalized Sylvester equations \eqref{eq:projected} have a unique solution $(\hat{X}, \hat{Y}) \in \mathbb{R}^{r \times r} \times \mathbb{R}^{r \times r}$ for every right-hand side.
\end{theorem}

\begin{proof}
Eliminating $\hat{Y}$ from the homogeneous projected system gives
\[
M_1 \hat{X} = \hat{X} M_2, \qquad
M_1 = \hat{E}_1^{-1} \hat{A}_1, \quad
M_2 = \hat{E}_2^{-1} \hat{A}_2.
\]
Since the projected operator is square, unique solvability for every right-hand side is equivalent to
$\Lambda(M_1) \cap \Lambda(M_2) = \emptyset$.

The definitions in \eqref{eq:test_bases} give
\[
E_1 \hat{V}_x = \hat{V}_y \hat{E}_1, \qquad
E_2^\top \hat{W}_y = \hat{W}_x \hat{E}_2^\top.
\]
Consequently,
\begin{align*}
\hat{A}_1 - \lambda \hat{E}_1
&= \hat{V}_y^\top (A_1 E_1^{-1} - \lambda I)
   E_1 \hat{V}_x
 = (K_1 - \lambda I)\hat{E}_1, \\
\hat{A}_2 - \lambda \hat{E}_2
&= \hat{W}_y^\top E_2 (E_2^{-1} A_2 - \lambda I)
   \hat{W}_x
 = \hat{E}_2(L_2 - \lambda I),
\end{align*}
where
\[
K_1 = \hat{V}_y^\top A_1 E_1^{-1} \hat{V}_y, \qquad
L_2 = \hat{W}_x^\top E_2^{-1} A_2 \hat{W}_x.
\]
The nonsingularity of $\hat{E}_1$ and $\hat{E}_2$ therefore implies
\[
\Lambda(M_1) = \Lambda(K_1), \qquad
\Lambda(M_2) = \Lambda(L_2).
\]
Field-of-values containment for orthogonal compressions \cite{simoncini2016computational} gives
\[
\Lambda(K_1) \subseteq \mathcal{F}(A_1 E_1^{-1}), \qquad
\Lambda(L_2) \subseteq \mathcal{F}(E_2^{-1} A_2).
\]
Thus, \eqref{eq:fov_condition} ensures that
$\Lambda(M_1) \cap \Lambda(M_2) = \emptyset$, completing the proof.
\end{proof}

\begin{remark}
Condition \eqref{eq:fov_condition} is sufficient but not necessary for unique solvability of \eqref{eq:projected}. The necessary and sufficient condition is
$\Lambda(M_1) \cap \Lambda(M_2) = \emptyset$, or, equivalently, coprimeness of the polynomials
$\det(\hat{A}_1 - s\hat{E}_1)$ and
$\det(\hat{A}_2 - s\hat{E}_2)$
\cite{wimmer1994consistency}.
\end{remark}
\subsection{Choice of Projection Matrices}

The choice of projection subspaces is guided by the structural properties of the exact solution pair.

\begin{proposition}\label{prop:svd_structure}
Let \(X,Y\in\mathbb{R}^{n_1\times n_2}\) uniquely solve the coupled Sylvester equations \eqref{eq:coupled}. Assume that \(X\) and \(Y\) admit exact thin singular value decompositions of rank \(r\):
\begin{equation}\label{eq:svd}
X = V_x\Sigma_xW_x^{\top},
\qquad
Y = V_y\Sigma_yW_y^{\top},
\end{equation}
where \(V_x,V_y\in\mathbb{R}^{n_1\times r}\) and \(W_x,W_y\in\mathbb{R}^{n_2\times r}\) have orthonormal columns, and \(\Sigma_x,\Sigma_y\in\mathbb{R}^{r\times r}\) are nonsingular diagonal matrices. Define the projected matrices with respect to these exact factors:
\begin{align*}
\bar{A}_1 &= V_y^{\top}A_1V_x, \quad \bar{E}_1 = V_y^{\top}E_1V_x,\quad \bar{B}_1 = V_y^{\top}B_1,\quad \bar{B}_2 = V_y^{\top}B_2,\\
\bar{A}_2 &= W_y^{\top}A_2W_x,\quad \bar{E}_2 = W_y^{\top}E_2W_x,\quad \bar{C}_1 = C_1W_x,\quad \bar{C}_2 = C_2W_x.
\end{align*}
Assume that \(\bar{E}_1, \bar{E}_2, \bar{A}_1,\) and \(\bar{A}_2\) are nonsingular. Further assume that \(V_{11}, V_{12}, W_{21},\) and \(W_{22}\) uniquely solve the following Sylvester equations:
\begin{align}
A_1V_{11} - E_1V_{11}\big(\bar{A}_2\bar{E}_2^{-1}\big) + B_1\big(\bar{C}_1\bar{E}_2^{-1}\big) &= 0, \label{eq:V11}\\
A_1V_{12} - E_1V_{12}\big(\bar{A}_2\bar{E}_2^{-1}\big) - B_2\big(\bar{C}_2\bar{E}_2^{-1}\big) &= 0, \label{eq:V12}\\
A_2^{\top}W_{21} - E_2^\top W_{21}\big(\bar{E}_1^{-1}\bar{A}_1\big)^\top + C_1^{\top}\big(\bar{E}_1^{-1}\bar{B}_1\big)^\top &= 0, \label{eq:W21}\\
A_2^{\top}W_{22} - E_2^\top W_{22}\big(\bar{E}_1^{-1}\bar{A}_1\big)^\top - C_2^{\top}\big(\bar{E}_1^{-1}\bar{B}_2\big)^{\top} &= 0. \label{eq:W22}
\end{align}
Then the SVD factors satisfy the following relations:
\begin{align}
V_x &= (V_{11}\bar{E}_2+V_{12}\bar{A}_2)\Sigma_x^{-1}, \label{eq:Vx}\\
W_x &= -\left(E_2^{\top}W_{21} + A_2^{\top}W_{22}\right)\Sigma_x^{-1}, \label{eq:Wx}\\
V_y &= -\left(E_1V_{11} + A_1V_{12}\right)\Sigma_y^{-1}, \label{eq:Vy}\\
W_y &= (W_{21}\bar{E}_1^{\top}+W_{22}\bar{A}_1^{\top})\Sigma_y^{-1}. \label{eq:Wy}
\end{align}
\end{proposition}
\begin{proof}
The proof is provided in the appendix.
\end{proof}

Proposition~\ref{prop:svd_structure} immediately yields the following subspace inclusions:
\begin{align}
\operatorname{Ran}(V_x) &\subseteq \operatorname{Ran}\left(\begin{bmatrix} V_{11} & V_{12} \end{bmatrix}\right), \label{eq:Vx_enlarged}\\
\operatorname{Ran}(V_y) &\subseteq \operatorname{Ran}\left(\begin{bmatrix} E_1V_{11} & A_1V_{12} \end{bmatrix}\right), \label{eq:Vy_enlarged}\\
\operatorname{Ran}(W_y) &\subseteq \operatorname{Ran}\left(\begin{bmatrix} W_{21} & W_{22} \end{bmatrix}\right), \label{eq:Wy_enlarged}\\
\operatorname{Ran}(W_x) &\subseteq \operatorname{Ran}\left(\begin{bmatrix} E_2^{\top}W_{21} & A_2^{\top}W_{22} \end{bmatrix}\right). \label{eq:Wx_enlarged}
\end{align}
These follow directly from the factorizations \eqref{eq:Vx}--\eqref{eq:Wy}. For instance, \eqref{eq:Vx} implies \(V_x\Sigma_x = \begin{bmatrix} V_{11} & V_{12} \end{bmatrix} \begin{bmatrix} \bar{E}_2 \\ \bar{A}_2 \end{bmatrix}\), which establishes \eqref{eq:Vx_enlarged}. The enlarged bases thus contain the ranges of the exact SVD factors.

To expose the underlying pole-residue structure, suppose temporarily that the reduced pencils \(\bar{E}_1^{-1}\bar{A}_1\) and \(\bar{A}_2\bar{E}_2^{-1}\) are diagonalizable. In suitable reduced coordinates, the projected matrices take the modal forms:
\begin{align*}
\bar{E}_1 &= I, & \bar{A}_1 &= \operatorname{diag}(\sigma_1,\dots,\sigma_r), & \bar{B}_1 &= \begin{bmatrix} b_{1,1}^\top \\ \vdots \\ b_{1,r}^\top \end{bmatrix}, & \bar{B}_2 &= \begin{bmatrix} b_{2,1}^\top \\ \vdots \\ b_{2,r}^\top \end{bmatrix}, \\
\bar{E}_2 &= I, & \bar{A}_2 &= \operatorname{diag}(\lambda_1,\dots,\lambda_r), & \bar{C}_1 &= \begin{bmatrix} c_{1,1} & \cdots & c_{1,r} \end{bmatrix}, & \bar{C}_2 &= \begin{bmatrix} c_{2,1} & \cdots & c_{2,r} \end{bmatrix},
\end{align*}
where \(b_{k,i}, c_{k,i} \in \mathbb{C}^m\). Under this representation, the solutions to \eqref{eq:V11}--\eqref{eq:W22} decouple into resolvent evaluations, yielding the refined inclusions:
\begin{align}
\operatorname{span}\left\{ (\lambda_i E_1 - A_1)^{-1} B_1 c_{1,i} \mid i=1,\dots,r \right\} &\subseteq \operatorname{Ran}(V_{11}), \label{eq:res_V11}\\
\operatorname{span}\left\{ (\lambda_i E_1 - A_1)^{-1} B_2 c_{2,i} \mid i=1,\dots,r \right\} &\subseteq \operatorname{Ran}(V_{12}), \label{eq:res_V12}\\
\operatorname{span}\left\{ (\sigma_i^* E_2^\top - A_2^\top)^{-1} C_1^{\top} b_{1,i}^* \mid i=1,\dots,r \right\} &\subseteq \operatorname{Ran}(W_{21}), \label{eq:res_W21}\\
\operatorname{span}\left\{ (\sigma_i^* E_2^\top - A_2^\top)^{-1} C_2^{\top} b_{2,i}^* \mid i=1,\dots,r \right\} &\subseteq \operatorname{Ran}(W_{22}). \label{eq:res_W22}
\end{align}

This pole-residue structure provides a clear control-theoretic interpretation of the solution's rank. The poles \(\{\sigma_i\}_{i=1}^r\) correspond to the controllable modes of the system \((E_1, A_1, [B_1 \ B_2])\), while \(\{\lambda_i\}_{i=1}^r\) correspond to the observable modes of \((E_2, A_2, [C_1^\top \ C_2^\top]^\top)\). Crucially, any strictly uncontrollable pole \(\{\sigma_i\}_{i=1}^{n_1}\) has a zero residue, which forces the corresponding column in the right singular factor \(W_y\) to vanish, thereby strictly reducing the column rank of \(Y\) below \(n_2\). Similarly, unobservable poles \(\{\lambda_i\}_{i=1}^{n_2}\) yield zero residues that force the column rank of \(V_x\) (and thus \(X\)) to be strictly less than \(n_1\). In practical large-scale applications where the system matrices are sparse and \(m \ll n_1, n_2\), the majority of the system poles are not strictly uncontrollable or unobservable, but rather \emph{nearly} so. Consequently, their associated residues are vanishingly small, leading to rapidly decaying singular values. This mechanism is the fundamental reason \(X\) and \(Y\) exhibit numerical low-rank behavior, forming the theoretical basis for applying model order reduction techniques to this problem.

Motivated by this structure, if we define the trial and test bases as the orthogonalizations of these enlarged ranges:
\begin{align*}
\hat{V}_x &= \operatorname{orth}\left(\begin{bmatrix} V_{11} & V_{12} \end{bmatrix}\right), & \hat{V}_y &= \operatorname{orth}\left(\begin{bmatrix} E_1V_{11} & A_1V_{12} \end{bmatrix}\right), \\
\hat{W}_x &= \operatorname{orth}\left(\begin{bmatrix} E_2^\top W_{21} & A_2^\top W_{22} \end{bmatrix}\right), & \hat{W}_y &= \operatorname{orth}\left(\begin{bmatrix} W_{21} & W_{22} \end{bmatrix}\right),
\end{align*}
and project the system as in \eqref{eq:projected}, the resulting reduced models satisfy the following tangential interpolation conditions:
\begin{align}
(\lambda_i E_1 - A_1)^{-1} \begin{bmatrix} B_1 & B_2 \end{bmatrix} \begin{bmatrix} c_{1,i} \\ c_{2,i} \end{bmatrix} &= \hat{V}_x (\lambda_i \hat{E}_1 - \hat{A}_1)^{-1} \begin{bmatrix} \hat{B}_1 & \hat{B}_2 \end{bmatrix} \begin{bmatrix} c_{1,i} \\ c_{2,i} \end{bmatrix}, \label{eq:interp_V}\\
(\sigma_i^* E_2^\top - A_2^\top)^{-1} \begin{bmatrix} C_1^{\top} & C_2^{\top} \end{bmatrix} \begin{bmatrix} b_{1,i}^* \\ b_{2,i}^* \end{bmatrix} &= \hat{W}_y (\sigma_i^* \hat{E}_2^\top - \hat{A}_2^\top)^{-1} \begin{bmatrix} \hat{C}_1^{\top} & \hat{C}_2^{\top} \end{bmatrix} \begin{bmatrix} b_{1,i}^* \\ b_{2,i}^* \end{bmatrix}. \label{eq:interp_W}
\end{align}
If \(\hat{X}\) and \(\hat{Y}\) uniquely solve this projected system, the lifted approximations \(\tilde{X} = \hat{V}_x\hat{X}\hat{W}_x^\top\) and \(\hat{Y} = \hat{V}_y\hat{Y}\hat{W}_y^\top\) recover \(X\) and \(Y\) exactly, because the ranges of the exact SVD factors are contained within the constructed bases.

While the exact poles \(\{\sigma_i, \lambda_i\}\) and residues are not known a priori, the preceding analysis demonstrates that constructing effective projection matrices reduces to a standard model reduction problem: finding bases \(\hat{V}_x\) and \(\hat{W}_y\) that approximate the dominant dynamics of the associated transfer functions, namely,
\begin{align}
(sE_1 - A_1)^{-1} \begin{bmatrix} B_1 & B_2 \end{bmatrix} &\approx \hat{V}_x (s\hat{E}_1 - \hat{A}_1)^{-1} \begin{bmatrix} \hat{B}_1 & \hat{B}_2 \end{bmatrix}, \label{eq:mor_V}\\
(sE_2^\top - A_2^\top)^{-1} \begin{bmatrix} C_1^{\top} & C_2^{\top} \end{bmatrix} &\approx \hat{W}_y (s\hat{E}_2^\top - \hat{A}_2^\top)^{-1} \begin{bmatrix} \hat{C}_1^{\top} & \hat{C}_2^{\top} \end{bmatrix}. \label{eq:mor_W}
\end{align}
This connection bridges the algebraic structure of the coupled Sylvester equations with rational interpolation techniques, guiding the algorithmic development in the subsequent sections.
\subsection{Extended Krylov-based Projection Method}
Extended Krylov subspace methods are effective at capturing dominant controllable and observable dynamics when computing low-rank solutions to large-scale matrix equations \cite{simoncini2016computational,simoncini2010extended}. From a rational interpolation perspective, these methods use fixed interpolation points at $s=0$ and $s=\infty$ to match coefficients in the Taylor expansions at the origin and the Laurent expansions at infinity of
\[
(sE_1-A_1)^{-1}\begin{bmatrix}B_1 & B_2\end{bmatrix}\quad \text{and}\quad
(sE_2^\top-A_2^\top)^{-1}
\begin{bmatrix}C_1^\top & C_2^\top\end{bmatrix}.
\]

The extended Krylov subspace method requires the additional assumption that \(A_1\) and \(A_2\) are invertible. The trial bases $\hat{V}_x$ and $\hat{W}_y$ satisfy
\begin{align}
\operatorname{Ran}(\hat{V}_x)
&\supseteq \operatorname{span}\left\{
(E_1^{-1}A_1)^{i-1}E_1^{-1}\tilde{B},\
(A_1^{-1}E_1)^{i-1}A_1^{-1}\tilde{B}
\mid i=1,\dots,k
\right\}, \nonumber\\
\operatorname{Ran}(\hat{W}_y)
&\supseteq \operatorname{span}\left\{
(E_2^{-\top}A_2^\top)^{i-1}E_2^{-\top}\tilde{C}^\top,\
(A_2^{-\top}E_2^\top)^{i-1}A_2^{-\top}\tilde{C}^\top
\mid i=1,\dots,k
\right\}, \label{eq:ext_krylov}
\end{align}
where $\tilde{B}=\begin{bmatrix}B_1 & B_2\end{bmatrix}$ and
$\tilde{C}^\top=\begin{bmatrix}C_1^\top & C_2^\top\end{bmatrix}$.
The test bases $\hat{V}_y$ and $\hat{W}_x$ are constructed as in Theorem~\ref{thm:projected}. The fixed interpolation points eliminate the need for shift selection, which can substantially affect the convergence of rational Krylov methods \cite{druskin2011analysis}. The resulting Extended Krylov-based Projection Algorithm for Coupled Generalized Sylvester Equations (EKPA-C-SYLV) is presented in Algorithm~\ref{alg1}.

\begin{algorithm}[!h]
\DontPrintSemicolon
\SetAlgoLined
\SetKwInOut{KwIn}{Inputs}
\SetKwInOut{KwOut}{Output}
\KwIn{Matrices of the Coupled Generalized Sylvester Equations \eqref{eq:coupled}: $(E_1,A_1,B_1,B_2,E_2,A_2,C_1,C_2)$; Tolerance: $\tau$.}
\KwOut{Low-rank approximation of $X$ and $Y$: $X\approx\tilde{X}=\hat{V}_x\hat{X}\hat{W}_x^\top$ and $Y\approx\tilde{Y}=\hat{V}_y\hat{Y}\hat{W}_y^\top$.}

Initialize: $\hat{V}_x = [\;]$, $\hat{V}_y = [\;]$, $\hat{W}_x = [\;]$, $\hat{W}_y = [\;]$, $\tilde{B}=[B_1,B_2]$, $\tilde{C}^\top=[C_1^\top,C_2^\top]$, $\text{rhs}_1=\tilde{B}$, $\text{rhs}_2=\tilde{B}$, $\text{rhs}_3=\tilde{C}^\top$, and $\text{rhs}_4=\tilde{C}^\top$.\;
\While{\Big($\frac{\|A_1\tilde{X} + \tilde{Y} A_2 +B_1C_1\|_F}{\|A_1\tilde{X}\|_F+\|\tilde{Y} A_2\|_F+\|B_1C_1\|_F} > \tau$\Big) $\lor$ \Big($\frac{\|E_1\tilde{X} + \tilde{Y} E_2 +B_2C_2\|_F}{\|E_1\tilde{X}\|_F+\|\tilde{Y} E_2\|_F+\|B_2C_2\|_F} > \tau$\Big)}{
    Solve for $v_{1}$, $v_{2}$, $w_{1}$, and $w_{2}$:\\
    $E_1v_{1} = \text{rhs}_{1}$, $A_1v_{2} = \text{rhs}_{2}$, $E_2^\top w_{1} = \text{rhs}_{3}$, and $A_2^\top w_{2} = \text{rhs}_{4}$.\;
    Orthogonally expand the trial bases: $\hat{V}_x \gets \mathrm{orth}\Big(\begin{bmatrix}\hat{V}_x & [v_{1}, v_{2}]\end{bmatrix}\Big)$ and $\hat{W}_y \gets \mathrm{orth}\Big(\begin{bmatrix}\hat{W}_y & [w_{1},w_{2}]\end{bmatrix}\Big)$.\;
    Orthogonally expand the test bases: $\hat{V}_y \gets \mathrm{orth}\Big(\begin{bmatrix}\hat{V}_y & E_1[v_{1},v_{2}]\end{bmatrix}\Big)$ and $\hat{W}_x \gets \mathrm{orth}\Big(\begin{bmatrix}\hat{W}_x & E_2^\top [w_{1}, w_{2}]\end{bmatrix}\Big)$.\;
    Project: $\hat{A}_1 \gets \hat{V}_y^{\top}A_1\hat{V}_x$, $\hat{E}_1 \gets \hat{V}_y^{\top}E_1\hat{V}_x$, $\hat{B}_1 \gets \hat{V}_y^{\top}B_1$, $\hat{B}_2 \gets \hat{V}_y^{\top}B_2$, $\hat{A}_2 \gets \hat{W}_y^{\top}A_2\hat{W}_x$, $\hat{E}_2 \gets \hat{W}_y^{\top}E_2\hat{W}_x$, $\hat{C}_1 \gets C_1\hat{W}_x$, and $\hat{C}_2 \gets C_2\hat{W}_x$.\;
    Solve the projected coupled generalized Sylvester equations \eqref{eq:projected} to compute $\hat{X}$ and $\hat{Y}$.\;
    Update: $\text{rhs}_1\gets A_1v_{1}$, $\text{rhs}_2\gets E_1v_{2}$, $\text{rhs}_3\gets A_2^\top w_{1}$, and $\text{rhs}_4\gets E_2^\top w_{2}$.
}
\caption{EKPA-C-SYLV}
\label{alg1}
\end{algorithm}

The main implementation details of EKPA-C-SYLV are summarized below:
\begin{enumerate}
  \item Stopping criterion: The iteration terminates when the normalized residuals of \eqref{eq:coupled} fall below the prescribed tolerance $\tau$. The tall-skinny basis matrices $\hat{V}_x$, $\hat{V}_y$, $\hat{W}_x$, and $\hat{W}_y$ allow the residual Frobenius norms to be evaluated efficiently.

  \item Trial bases: The efficiency of the extended Krylov method relies on exploiting the sparsity of the coefficient matrices. Since $E_1$, $A_1$, $E_2^\top$, and $A_2^\top$ remain fixed, their sparse factorizations are computed once before the iteration and reused in subsequent linear solves. Orthogonalization, denoted by $\mathrm{orth}(\cdot)$, is performed using modified Gram--Schmidt with reorthogonalization to maintain numerical orthogonality.

  \item Test bases: The test bases are constructed by orthonormalizing $E_1\hat{V}_x$ and $E_2^\top\hat{W}_y$. Under the standing assumptions, this construction ensures that the projected mass matrices are nonsingular. The field-of-values separation condition in Theorem~\ref{thm:projected} further guarantees that the projected operator is bijective.

  \item Projected solutions: At each iteration, the dense, reduced-order coupled generalized Sylvester equations \eqref{eq:projected} is solved using the direct algorithm of Shahzad et al.~\cite{shahzad2011efficient}.
\end{enumerate}
\subsection{Rational Krylov-based Projection Method}
Unlike extended Krylov methods, which use fixed interpolation points at $s=0$ and $s=\infty$, rational Krylov methods allow flexible shift selection. Their performance depends strongly on the choice of shifts, but effective shifts are generally difficult to determine \cite{druskin2011analysis}. The proposed Rational Krylov-based Projection Algorithm for Coupled Generalized Sylvester Equations (RKPA-C-SYLV) selects shifts automatically, without user intervention. The algorithm is presented in Algorithm~\ref{alg2}.

\begin{algorithm}[!h]
\DontPrintSemicolon
\SetAlgoLined
\SetKwInOut{KwIn}{Inputs}
\SetKwInOut{KwOut}{Output}
\KwIn{Matrices of the Coupled Generalized Sylvester Equations \eqref{eq:coupled}: $(E_1,A_1,B_1,B_2,E_2,A_2,C_1,C_2)$; Tolerance: $\tau$; Type: \textit{block} or \textit{tangential}.}
\KwOut{Low-rank approximation of $X$ and $Y$: $X\approx\tilde{X}=\hat{V}_x\hat{X}\hat{W}_x^\top$ and $Y\approx\tilde{Y}=\hat{V}_y\hat{Y}\hat{W}_y^\top$.}

Initialize: $\hat{V}_x = [\;]$, $\hat{V}_y = [\;]$, $\hat{W}_x = [\;]$, $\hat{W}_y = [\;]$, $B_{\perp}=[B_1,B_2]$, $C_{\perp}^\top=[C_1^\top,C_2^\top]$, $\alpha_1=\alpha_{\text{initial}}$, $r_1=\mathbf{1}_{1\times 2m}$, $\beta_1=\beta_{\text{initial}}$, $t_1=\mathbf{1}_{2m\times 1}$, $i=1$.\;
\While{\Big($\frac{\|A_1\tilde{X} + \tilde{Y} A_2 +B_1C_1\|_F}{\|A_1\tilde{X}\|_F+\|\tilde{Y} A_2\|_F+\|B_1C_1\|_F} > \tau$\Big) $\lor$ \Big($\frac{\|E_1\tilde{X} + \tilde{Y} E_2 +B_2C_2\|_F}{\|E_1\tilde{X}\|_F+\|\tilde{Y} E_2\|_F+\|B_2C_2\|_F} > \tau$\Big)}{
 \eIf{\textnormal{Type} = ``\textit{block}''}{
        $\mathcal{B} = B_{\perp}$ and $\mathcal{C} = C_{\perp}$.\;
    }{
        $\mathcal{B} = \begin{bmatrix}B_{\perp}(:,1:m) t_i(1:m) & B_{\perp}(:,m+1:2m) t_i(m+1:2m)\end{bmatrix}$ and
        $\mathcal{C}^\top = \begin{bmatrix}C_{\perp}(1:m,:)^\top r_i(1:m)^\top & C_{\perp}(m+1:2m,:)^\top r_i(m+1:2m)^\top\end{bmatrix}$.\;
    }
    Solve for $[v_{1}, v_{2}]$ and $[w_{1}, w_{2}]$:\\
    $\big(\beta_iE_1-A_1\big)[v_{1}, v_{2}] = \mathcal{B}$ and $\big(\alpha_i E_2^\top - A_2^\top\big)[w_{1}, w_{2}] = \mathcal{C}^\top$.\;
   Orthogonally expand the trial bases: $\hat{V}_x \gets \mathrm{orth}\Big(\begin{bmatrix}\hat{V}_x & [v_{1},v_{2}]\end{bmatrix}\Big)$ and $\hat{W}_y \gets \mathrm{orth}\Big(\begin{bmatrix}\hat{W}_y & [w_{1}, w_{2}]\end{bmatrix}\Big)$.\;
     Orthogonally expand the test bases: $\hat{V}_y \gets \mathrm{orth}\Big(\begin{bmatrix}\hat{V}_y & E_1[v_{1},v_{2}]\end{bmatrix}\Big)$ and $\hat{W}_x \gets \mathrm{orth}\Big(\begin{bmatrix}\hat{W}_x & E_2^\top [w_{1},w_{2}]\end{bmatrix}\Big)$.\;
    Project: $\hat{A}_1 \gets \hat{V}_y^{\top}A_1\hat{V}_x$, $\hat{E}_1 \gets \hat{V}_y^{\top}E_1\hat{V}_x$, $\hat{B}_1 \gets \hat{V}_y^{\top}B_1$, $\hat{B}_2 \gets \hat{V}_y^{\top}B_2$, $\hat{A}_2 \gets \hat{W}_y^{\top}A_2\hat{W}_x$, $\hat{E}_2 \gets \hat{W}_y^{\top}E_2\hat{W}_x$, $\hat{C}_1 \gets C_1\hat{W}_x$, and $\hat{C}_2 \gets C_2\hat{W}_x$.\;
    Solve the projected coupled generalized Sylvester equations \eqref{eq:projected} to compute $\hat{X}$ and $\hat{Y}$.\;
    Deflate: $B_{\perp}\gets \begin{bmatrix}B_1&B_2\end{bmatrix}-E_1\hat{V}_x\hat{E}_1^{-1}\begin{bmatrix}\hat{B}_1&\hat{B}_2\end{bmatrix}$ and $C_{\perp}\gets \begin{bmatrix}C_1\\C_2\end{bmatrix}-\begin{bmatrix}\hat{C}_1\\\hat{C}_2\end{bmatrix}\hat{E}_2^{-1}\hat{W}_y^{\top}E_2$.\;
    Compute the Ritz values: $(\hat{V}_x^\top E_1\hat{V}_x)^{-1}(\hat{V}_x^\top A_1 \hat{V}_x)=\hat{T}_1\begin{bsmallmatrix}\tilde{\alpha}_1&\cdots&0\\\vdots&\ddots&\vdots\\0&\cdots&\tilde{\alpha}_r\end{bsmallmatrix}\hat{T}_1^{-1}$ and $(\hat{W}_y^\top A_2\hat{W}_y)(\hat{W}_y^\top E_2\hat{W}_y)^{-1}=\hat{T}_2\begin{bsmallmatrix}\tilde{\beta}_1&\cdots&0\\\vdots&\ddots&\vdots\\0&\cdots&\tilde{\beta}_r\end{bsmallmatrix}\hat{T}_2^{-1}$.\;
    Compute the residues: $b_k=\hat{T}_1^{-1}(k,:)(\hat{V}_x^\top E_1\hat{V}_x)^{-1}(\hat{V}_x^\top B_{\perp})$ and $c_k=(C_{\perp}\hat{W}_y)(\hat{W}_y^\top E_2\hat{W}_y)^{-1}\hat{T}_2(:,k)$.\;
    Sort $\tilde{\alpha}_k$ and $\tilde{\beta}_k$ and their respective residues $b_k$ and $c_k$ in the descending order of their respective normalized residuals $\frac{\|b_k\|_2^2}{|\mathrm{Re}(\tilde{\alpha}_k)|}$ and $\frac{\|c_k\|_2^2}{|\mathrm{Re}(\tilde{\beta}_k)|}$.\;
    Generate the interpolation data: $\alpha_i=\tilde{\alpha}_1$, $r_i=(\frac{b_1}{\|b_1\|_2})^\top$, $\beta_i=\tilde{\beta}_1$, $t_i=\frac{c_1}{\|c_1\|_2}$, $i\gets i+1$.
}
\caption{RKPA-C-SYLV}
\label{alg2}
\end{algorithm}

The main implementation details of RKPA-C-SYLV are summarized below:
\begin{enumerate}
  \item Trial bases: The trial bases are chosen based on the structural properties established in Proposition~\ref{prop:svd_structure}.

  \item Shift generation: Since the controllable poles of $(E_1,A_1,[B_1,B_2])$ and $(E_2^\top,A_2^\top,[C_1^\top,C_2^\top])$ are unknown, RKPA-C-SYLV estimates them using the Ritz values of the reduced pencils $A_2E_2^{-1}$ and $E_1^{-1}A_1$, respectively. Dominant Ritz values are selected using the same measure employed in dominant pole estimation algorithms such as those in \cite{rommes2008modal,mengi2022large}.

  \item Deflation: The algorithm combines deflation and residue normalization to avoid redundant updates and retain useful expansion directions. Deflation removes the components of the right-hand sides already captured by the current subspaces, preventing repeated interpolation data. The residue vectors $b_1$ and $c_1$ are normalized to unit length so that directional updates can be extracted even when the selected Ritz values are weakly controllable or observable. Such values often occur near convergence, when the unnormalized residue norms may be negligible.

  \item Expansion type: The \textit{Type} input determines the number of expansion directions. In \textit{block} mode, the full right-hand side matrices $B_{\perp}$ and $C_{\perp}$ are used, each with $2m$ columns. In \textit{tangential} mode, each update has two columns. The update matrices $\mathcal{B}$ and $\mathcal{C}$ are formed from $B_{\perp}$ and $C_{\perp}$ using the normalized residue directions $r_i$ and $t_i$ associated with the dominant Ritz values. Tangential updates are particularly useful when $m$ is not small, since block updates can rapidly increase the projection dimensions and make the solution of \eqref{eq:projected} computationally expensive or infeasible.
\end{enumerate}
\section{Numerical Results}
This section evaluates the numerical performance of EKPA-C-SYLV and RKPA-C-SYLV. The MATLAB code for reproducing the results is publicly available \cite{mycodes}. All computations were performed in MATLAB R2025b on a Windows 11 laptop with 32 GB of RAM and an Intel Core Ultra 9 285H processor running at 2.9 GHz.

In all examples, $(E_1,A_1,[B_1,B_2])$ and $(E_2^\top,A_2^\top,[C_1^\top,C_2^\top])$ have $20$ and $10$ dominant poles, respectively. The remaining poles are nondominant, and the examples are constructed so that $X$ and $Y$ are numerically low-rank. The stopping tolerance is set to $\tau=10^{-8}$. The accuracy of both methods is assessed using the maximum normalized residual
\[
R=\max\{R_1,R_2\},
\]
where
\begin{align*}
R_1 =
\frac{
\|A_1\tilde{X}+\tilde{Y}A_2+B_1C_1\|_F
}{
\|A_1\tilde{X}\|_F+\|\tilde{Y}A_2\|_F+\|B_1C_1\|_F
}\quad \text{and}\quad
R_2 =
\frac{
\|E_1\tilde{X}+\tilde{Y}E_2+B_2C_2\|_F
}{
\|E_1\tilde{X}\|_F+\|\tilde{Y}E_2\|_F+\|B_2C_2\|_F
}.
\end{align*}
The initial shifts for RKPA-C-SYLV in all experiments are set to $(\alpha_{\text{initial}},\beta_{\text{initial}})=(0,0)$.
\subsection{Model Description}\label{sub}
The test models are constructed from $n_r$ distinct real poles $\eta_i$ and $n_c$ distinct complex-conjugate pole pairs, represented by $\zeta_i$ and $\overline{\zeta_i}$. For the complex poles, arbitrary blocks $b_i\in\mathbb{R}^{2\times m}$ and $c_i\in\mathbb{R}^{p\times 2}$ are chosen, and we define
\[
\alpha_i=
\begin{cases}
\displaystyle
\sqrt{\frac{10|\operatorname{Re}(\zeta_i)|}{\|c_i b_i\|_2}},
& i=1,\ldots,n_d,\\[2ex]
\displaystyle
\sqrt{\frac{0.1|\operatorname{Re}(\zeta_i)|}{\|c_i b_i\|_2}},
& i=n_d+1,\ldots,n_c.
\end{cases}
\]
The scaled input and output matrices are
\[
\mathbb{B}_1=
\begin{bmatrix}
\alpha_1 b_1\\
\vdots\\
\alpha_{n_c}b_{n_c}
\end{bmatrix}
\in\mathbb{R}^{2n_c\times m},
\qquad
\mathbb{C}_1=
\begin{bmatrix}
\alpha_1 c_1 & \cdots & \alpha_{n_c}c_{n_c}
\end{bmatrix}
\in\mathbb{R}^{p\times 2n_c}.
\]
This scaling assigns greater weight to the first $n_d$ complex-conjugate pairs, which constitute the $2n_d$ dominant poles.

For the real poles, we define
\[
\mathbb{B}_2=
0.001\left(\min_{i=1,\ldots,n_r}|\eta_i|\right)^2
\mathbf{1}_{n_r\times m},
\qquad
\mathbb{C}_2=
0.001\left(\min_{i=1,\ldots,n_r}|\eta_i|\right)^2
\mathbf{1}_{p\times n_r}.
\]
The corresponding state matrices are
\[
\mathbb{A}_1=
\operatorname{blkdiag}(J_1,\ldots,J_{n_c}),
\qquad
J_i=
\begin{bmatrix}
\operatorname{Re}(\zeta_i) & \operatorname{Im}(\zeta_i)\\
-\operatorname{Im}(\zeta_i) & \operatorname{Re}(\zeta_i)
\end{bmatrix},
\qquad
\mathbb{A}_2=\operatorname{diag}(\eta_1,\ldots,\eta_{n_r}).
\]

Let $n=n_r+2n_c$ and introduce the tridiagonal matrices
\[
\mathbb{P}=\operatorname{tridiag}\left(\frac{1}{2},1,\frac{1}{3}\right),
\qquad
\mathbb{Q}=\operatorname{tridiag}\left(\frac{1}{4},1,\frac{1}{6}\right),
\qquad
\mathbb{P},\mathbb{Q}\in\mathbb{R}^{n\times n}.
\]
A state-space realization of order $n$, with $m$ inputs and $p$ outputs, is then given by
\[
E=\mathbb{Q}\mathbb{P},
\qquad
A=\mathbb{Q}
\begin{bmatrix}
\mathbb{A}_1 & \mathbf{0}\\
\mathbf{0} & \mathbb{A}_2
\end{bmatrix}
\mathbb{P},
\qquad
B=\mathbb{Q}
\begin{bmatrix}
\mathbb{B}_1\\
\mathbb{B}_2
\end{bmatrix},
\qquad
C=
\begin{bmatrix}
\mathbb{C}_1 & \mathbb{C}_2
\end{bmatrix}
\mathbb{P}.
\]
The matrices $E$ and $A$ are sparse by construction. The real poles are chosen to be nearly uncontrollable and unobservable, while the first $n_d$ complex-conjugate pairs are the most controllable and observable.

The model order, pole locations, and numbers of inputs and outputs are user-specified, providing benchmark models with prescribed properties for testing model order reduction algorithms and low-rank matrix-equation solvers. The poles are known by construction, even when an eigenvalue decomposition of $E^{-1}A$ is computationally infeasible. The MATLAB function \texttt{gen\_model}, available in \cite{mycodes}, implements this construction.
\subsection{Example 1}
This example considers a coupled generalized Sylvester system of moderate size for which the direct algorithm in \cite{shahzad2011efficient} remains feasible. The direct solution was used as a reference to assess the accuracy of the low-rank approximations computed by EKPA-C-SYLV and RKPA-C-SYLV.

The matrices $E_1,A_1\in\mathbb{R}^{1000\times 1000}$,
$B_1,B_2\in\mathbb{R}^{1000\times 1}$,
$E_2,A_2\in\mathbb{R}^{500\times 500}$, and
$C_1,C_2\in\mathbb{R}^{1\times 500}$ were generated using the procedure in Subsection~\ref{sub}. Since $m=1$, the \textit{block} and \textit{tangential} modes are equivalent. Figure~\ref{fig1} shows the decay of the normalized residual $R$. Both methods converged rapidly.
\begin{figure}[!h]
  \centering
  \includegraphics[width=12cm]{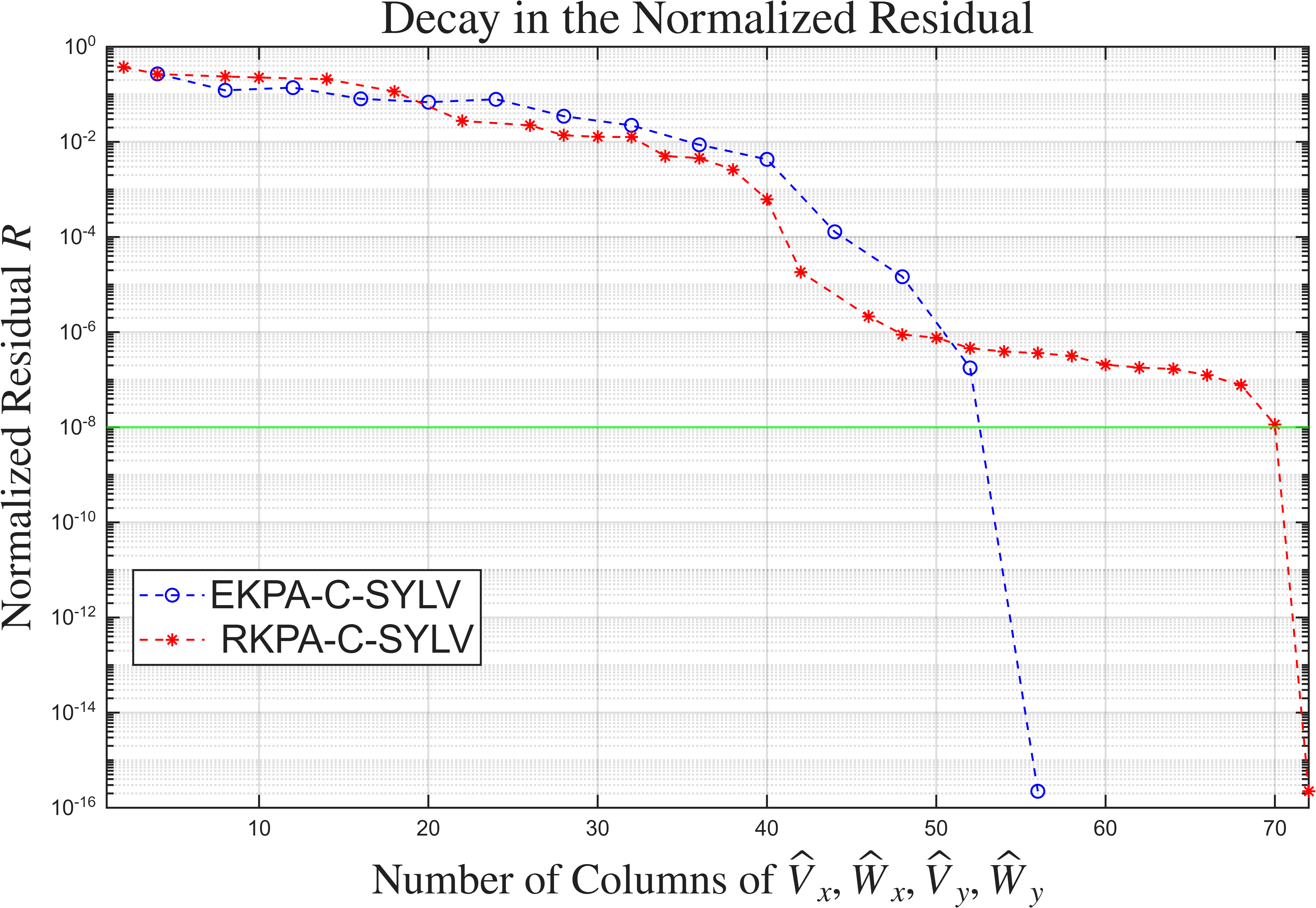}
  \caption{Decay of the normalized residual $R$.}
  \label{fig1}
\end{figure}

Let $(X_{\text{exact}},Y_{\text{exact}})$ denote the reference solution, and let $(X_{\text{ek}},Y_{\text{ek}})$ and $(X_{\text{rk}},Y_{\text{rk}})$ denote the low-rank approximations computed by EKPA-C-SYLV and RKPA-C-SYLV, respectively. The relative differences were
\begin{align*}
\frac{\|X_{\text{exact}}-X_{\text{ek}}\|_2}{\|X_{\text{exact}}\|_2}
&=1.8089\times 10^{-8},
&
\frac{\|Y_{\text{exact}}-Y_{\text{ek}}\|_2}{\|Y_{\text{exact}}\|_2}
&=5.5052\times 10^{-8},\\
\frac{\|X_{\text{exact}}-X_{\text{rk}}\|_2}{\|X_{\text{exact}}\|_2}
&=2.0141\times 10^{-8},
&
\frac{\|Y_{\text{exact}}-Y_{\text{rk}}\|_2}{\|Y_{\text{exact}}\|_2}
&=3.6828\times 10^{-8},
\end{align*}
and the corresponding absolute differences were
\begin{align*}
\|X_{\text{exact}}-X_{\text{ek}}\|_2
&=1.5651\times10^{-5},
&
\|Y_{\text{exact}}-Y_{\text{ek}}\|_2
&=1.5088\times10^{-5},\\
\|X_{\text{exact}}-X_{\text{rk}}\|_2
&=1.7425\times10^{-5},
&
\|Y_{\text{exact}}-Y_{\text{rk}}\|_2
&=1.0094\times10^{-5}.
\end{align*}
Both methods produced low-rank approximations with relative differences on the order of $10^{-8}$ from the reference solution.
\subsection{Example 2}
This example considers a large-scale coupled generalized Sylvester system for which the direct algorithm in \cite{shahzad2011efficient} is computationally infeasible. The matrices
$E_1,A_1,E_2,A_2\in\mathbb{R}^{10^5\times 10^5}$,
$B_1,B_2\in\mathbb{R}^{10^5\times 2}$, and
$C_1,C_2\in\mathbb{R}^{2\times 10^5}$ were generated using the procedure in Subsection~\ref{sub}.

Figure~\ref{fig2} shows the decay of the normalized residual $R$. Tangential updates offered no advantage in this example, as they slowed convergence. The adaptive shifts in RKPA-C-SYLV, however, yielded lower-rank approximations than EKPA-C-SYLV.
\begin{figure}[!h]
  \centering
  \includegraphics[width=12cm]{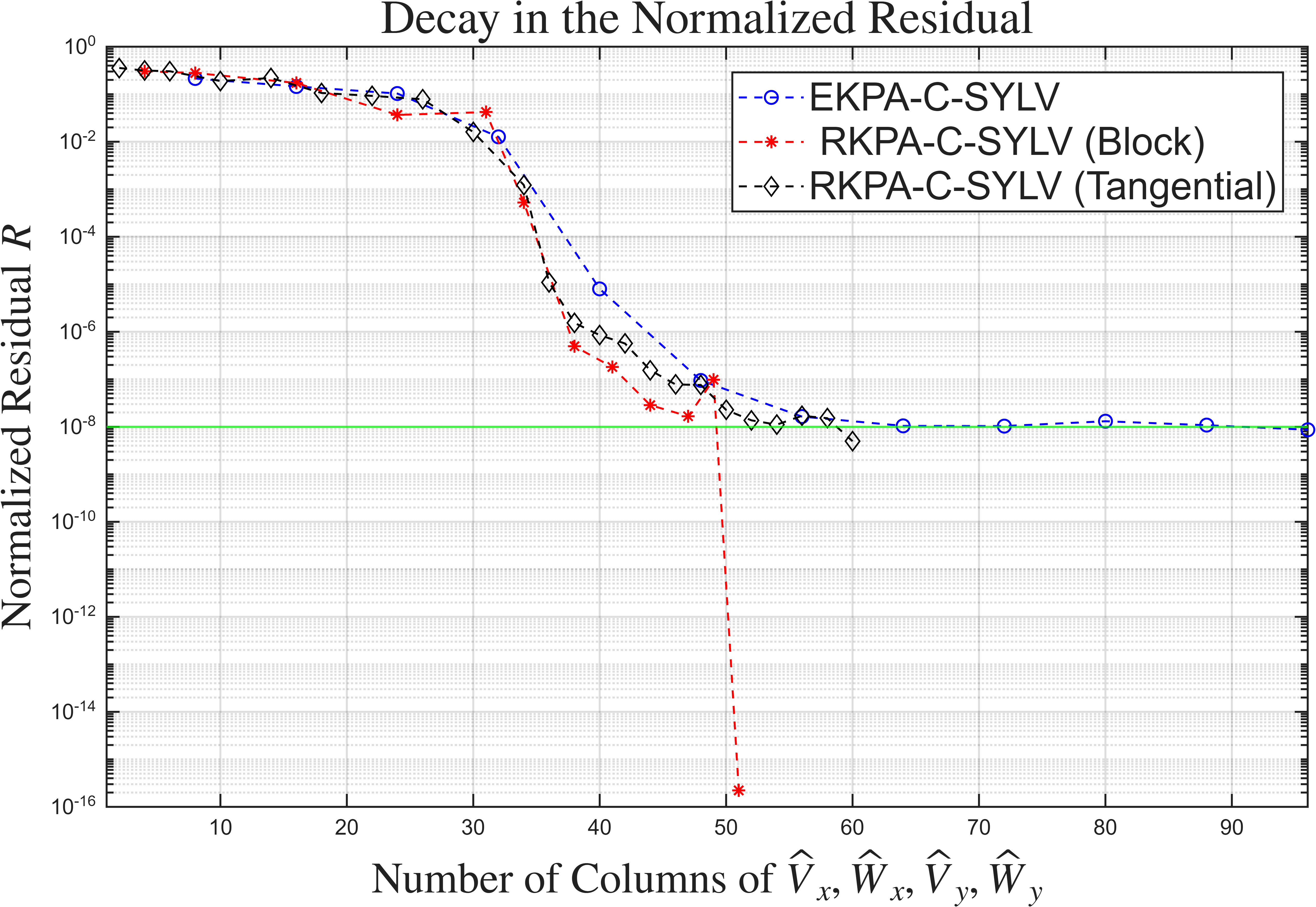}
  \caption{Decay of the normalized residual $R$.}
  \label{fig2}
\end{figure}

Table~\ref{tab1} compares the elapsed times. The block version of RKPA-C-SYLV was the fastest, and all three methods converged in less than $15$ seconds. Although EKPA-C-SYLV produced less compact approximations, it remained competitive in runtime because its sparse LU factorizations were computed only once and reused throughout the iteration. In contrast, RKPA-C-SYLV required new factorizations at each iteration.
\begin{table}[h]
  \centering
  \caption{Elapsed times for Example 2.}
  \begin{tabular}{lc}
    \hline
    Method & Elapsed time (s) \\
    \hline
    EKPA-C-SYLV & 9.4408 \\
    RKPA-C-SYLV (Block) & 7.2010 \\
    RKPA-C-SYLV (Tangential) & 14.2224 \\
    \hline
  \end{tabular}
  \label{tab1}
\end{table}
\subsection{Example 3}
This example considers a large-scale coupled generalized Sylvester system for which the direct algorithm in \cite{shahzad2011efficient} is computationally infeasible. The matrices
$E_1,A_1,E_2,A_2\in\mathbb{R}^{10^6\times 10^6}$,
$B_1,B_2\in\mathbb{R}^{10^6\times 5}$, and
$C_1,C_2\in\mathbb{R}^{5\times 10^6}$ were generated using the procedure in Subsection~\ref{sub}.

Figure~\ref{fig3} shows the decay of the normalized residual $R$. Tangential updates were particularly effective in this example: RKPA-C-SYLV converged and produced the most compact approximation. In this mode, at most two columns were added to each basis per iteration. In block mode, RKPA-C-SYLV could add up to $10$ columns per iteration but failed to converge and was terminated prematurely because of linear dependence in the bases. EKPA-C-SYLV could add up to $20$ columns per iteration and converged rapidly despite the faster growth of its bases.
\begin{figure}[!h]
  \centering
  \includegraphics[width=12cm]{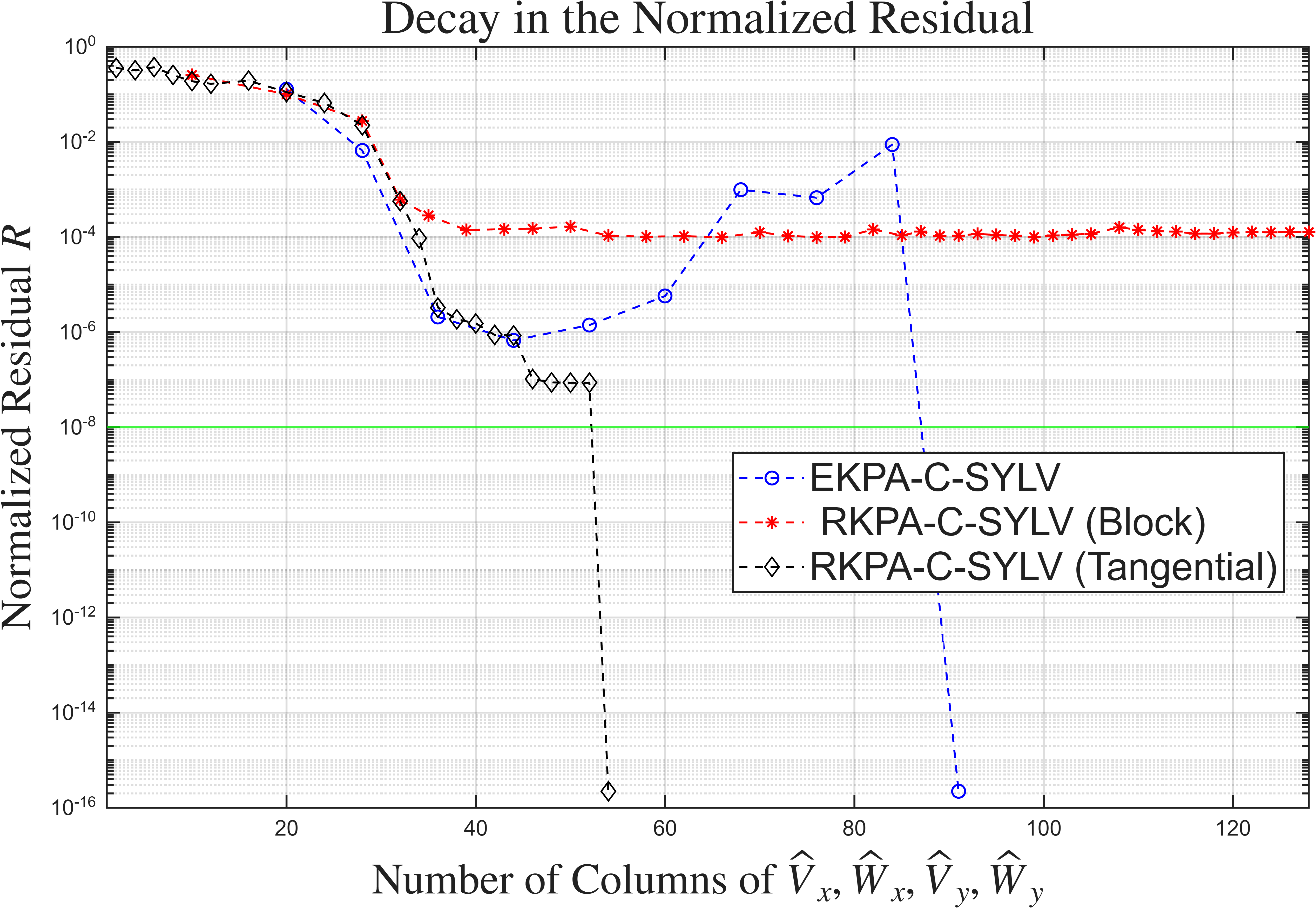}
  \caption{Decay of the normalized residual $R$.}
  \label{fig3}
\end{figure}

Table~\ref{tab2} compares the elapsed times. EKPA-C-SYLV was the fastest despite producing less compact approximations, as its sparse LU factorizations were computed once and reused throughout the iteration. In contrast, RKPA-C-SYLV required new factorizations at each iteration. The block version of RKPA-C-SYLV took the longest, struggling to refine the shifts after the bases reached $50$ columns. Subsequent updates failed to add useful information, eventually leading to linear dependence and termination.
\begin{table}[h]
  \centering
  \caption{Elapsed times for Example 3. The block version of RKPA-C-SYLV terminated without convergence.}
  \begin{tabular}{lc}
    \hline
    Method & Elapsed time (s) \\
    \hline
    EKPA-C-SYLV & 92.3005 \\
    RKPA-C-SYLV (Block) & 650.1552 \\
    RKPA-C-SYLV (Tangential) & 131.9391 \\
    \hline
  \end{tabular}
  \label{tab2}
\end{table}
\section{Conclusion}
This paper presented projection methods for computing low-rank solutions to large-scale coupled generalized Sylvester equations. The dominant subspaces of the exact solutions were shown to be contained in subspaces spanned by specific resolvent directions, providing the theoretical basis for the proposed methods. Within a Petrov--Galerkin framework, unique solvability of the projected equations was established under a field-of-values separation condition. This framework led to two algorithms: EKPA-C-SYLV and RKPA-C-SYLV. The former uses fixed shifts at $s=0$ and $s=\infty$, whereas the latter adaptively selects shifts from the Ritz values of the reduced pencils. Numerical experiments on large-scale problems demonstrated that the proposed algorithms computed accurate low-rank solutions efficiently.
\section*{Appendix}
\begin{proof}
Right-multiply the two equations in \eqref{eq:coupled} by \(W_x\). Using
\[
X = V_x\Sigma_xW_x^{\top},
\qquad
Y = V_y\Sigma_yW_y^{\top},
\qquad
W_x^{\top}W_x = I_r,
\]
and the definitions of the projected matrices, we obtain
\begin{align}
A_1V_x\Sigma_x + V_y\Sigma_y\bar{A}_2
&=
-B_1\bar{C}_1,
\label{eq:proj-left1}
\\
E_1V_x\Sigma_x + V_y\Sigma_y\bar{E}_2
&=
-B_2\bar{C}_2.
\label{eq:proj-left2}
\end{align}
Right-multiplying \eqref{eq:proj-left1} and \eqref{eq:proj-left2} by \(\bar{E}_2^{-1}\) gives
\begin{align}
A_1V_x\Sigma_x\bar{E}_2^{-1}
+
V_y\Sigma_y\bar{A}_2\bar{E}_2^{-1}
&=
-B_1\bar{C}_1\bar{E}_2^{-1},
\label{eq:proj-left1-inv}
\\
E_1V_x\Sigma_x\bar{E}_2^{-1}
+
V_y\Sigma_y
&=
-B_2\bar{C}_2\bar{E}_2^{-1}.
\label{eq:proj-left2-inv}
\end{align}
From \eqref{eq:proj-left2-inv},
\[
V_y\Sigma_y
=
-B_2\bar{C}_2\bar{E}_2^{-1}
-
E_1V_x\Sigma_x\bar{E}_2^{-1}.
\]
Substituting this into \eqref{eq:proj-left1-inv} yields
\begin{equation}\label{eq:V-main}
A_1V_x\Sigma_x\bar{E}_2^{-1}
-
E_1V_x\Sigma_x\bar{E}_2^{-1}
\left(\bar{A}_2\bar{E}_2^{-1}\right)
+
B_1\left(\bar{C}_1\bar{E}_2^{-1}\right)
-
B_2\left(\bar{C}_2\bar{E}_2^{-1}\right)
\left(\bar{A}_2\bar{E}_2^{-1}\right)
=
0.
\end{equation}

Now multiply \eqref{eq:V12} on the right by \(\bar{A}_2\bar{E}_2^{-1}\). This gives
\[
A_1V_{12}\bar{A}_2\bar{E}_2^{-1}
-
E_1V_{12}
\left(\bar{A}_2\bar{E}_2^{-1}\right)
\left(\bar{A}_2\bar{E}_2^{-1}\right)
-
B_2\left(\bar{C}_2\bar{E}_2^{-1}\right)
\left(\bar{A}_2\bar{E}_2^{-1}\right)
=
0.
\]
Adding this to \eqref{eq:V11} yields
\begin{align*}
& A_1\left(V_{11}+V_{12}\bar{A}_2\bar{E}_2^{-1}\right)
-
E_1\left(V_{11}+V_{12}\bar{A}_2\bar{E}_2^{-1}\right)
\left(\bar{A}_2\bar{E}_2^{-1}\right)
\\
&\qquad
+
B_1\left(\bar{C}_1\bar{E}_2^{-1}\right)
-
B_2\left(\bar{C}_2\bar{E}_2^{-1}\right)
\left(\bar{A}_2\bar{E}_2^{-1}\right)
=
0.
\end{align*}
Equivalently, since
\[
V_{11}+V_{12}\bar{A}_2\bar{E}_2^{-1}
=
\left(V_{11}\bar{E}_2+V_{12}\bar{A}_2\right)\bar{E}_2^{-1},
\]
the matrix
\[
\left(V_{11}\bar{E}_2+V_{12}\bar{A}_2\right)\bar{E}_2^{-1}
\]
satisfies the same Sylvester equation \eqref{eq:V-main} as
\[
V_x\Sigma_x\bar{E}_2^{-1}.
\]

The equations \eqref{eq:V-main} and \eqref{eq:V11} have the same left-hand Sylvester operator
\[
Z
\mapsto
A_1Z-E_1Z\left(\bar{A}_2\bar{E}_2^{-1}\right).
\]
Because the solution \(V_{11}\) of \eqref{eq:V11} is unique, the corresponding homogeneous equation has only the zero solution; otherwise adding a nonzero homogeneous solution to \(V_{11}\) would give a second solution of \eqref{eq:V11}. Hence \eqref{eq:V-main} has at most one solution. Therefore
\[
V_x\Sigma_x\bar{E}_2^{-1}
=
\left(V_{11}\bar{E}_2+V_{12}\bar{A}_2\right)\bar{E}_2^{-1}.
\]
Multiplying on the right by \(\bar{E}_2\) gives
\[
V_x\Sigma_x
=
V_{11}\bar{E}_2+V_{12}\bar{A}_2.
\]
Multiplying on the right by \(\Sigma_x^{-1}\) proves
\[
V_x
=
\left(V_{11}\bar{E}_2+V_{12}\bar{A}_2\right)\Sigma_x^{-1},
\]
which is \eqref{eq:Vx}.

Next, using \eqref{eq:proj-left2} again,
\[
V_y\Sigma_y\bar{E}_2
=
-B_2\bar{C}_2
-
E_1V_x\Sigma_x.
\]
Substitute the already proved identity
\[
V_x\Sigma_x
=
V_{11}\bar{E}_2+V_{12}\bar{A}_2
\]
into this expression:
\[
V_y\Sigma_y\bar{E}_2
=
-B_2\bar{C}_2
-
E_1V_{11}\bar{E}_2
-
E_1V_{12}\bar{A}_2.
\]
Multiplying \eqref{eq:V12} on the right by \(\bar{E}_2\) gives
\[
A_1V_{12}\bar{E}_2
-
E_1V_{12}\bar{A}_2
-
B_2\bar{C}_2
=
0.
\]
Hence
\[
-B_2\bar{C}_2
-
E_1V_{12}\bar{A}_2
=
-A_1V_{12}\bar{E}_2.
\]
Therefore
\[
V_y\Sigma_y\bar{E}_2
=
-\left(E_1V_{11}+A_1V_{12}\right)\bar{E}_2.
\]
Since \(\bar{E}_2\) is invertible,
\[
V_y\Sigma_y
=
-\left(E_1V_{11}+A_1V_{12}\right).
\]
Multiplying on the right by \(\Sigma_y^{-1}\) proves
\[
V_y
=
-\left(E_1V_{11}+A_1V_{12}\right)\Sigma_y^{-1},
\]
which is \eqref{eq:Vy}.

Now transpose the two equations in \eqref{eq:coupled}:
\[
X^{\top}A_1^{\top}+A_2^{\top}Y^{\top}
=
-C_1^{\top}B_1^{\top},
\qquad
X^{\top}E_1^{\top}+E_2^{\top}Y^{\top}
=
-C_2^{\top}B_2^{\top}.
\]
Using
\[
X^{\top}=W_x\Sigma_xV_x^{\top},
\qquad
Y^{\top}=W_y\Sigma_yV_y^{\top},
\]
and right-multiplying by \(V_y\), while using \(V_y^{\top}V_y=I_r\), we obtain
\begin{align}
W_x\Sigma_x\bar{A}_1^{\top}
+
A_2^{\top}W_y\Sigma_y
&=
-C_1^{\top}\bar{B}_1^{\top},
\label{eq:proj-right1}
\\
W_x\Sigma_x\bar{E}_1^{\top}
+
E_2^{\top}W_y\Sigma_y
&=
-C_2^{\top}\bar{B}_2^{\top}.
\label{eq:proj-right2}
\end{align}
Right-multiplying \eqref{eq:proj-right1} and \eqref{eq:proj-right2} by \(\bar{E}_1^{-\top}\), and using
\[
\bar{A}_1^{\top}\bar{E}_1^{-\top}
=
\left(\bar{E}_1^{-1}\bar{A}_1\right)^{\top},
\qquad
\bar{B}_k^{\top}\bar{E}_1^{-\top}
=
\left(\bar{E}_1^{-1}\bar{B}_k\right)^{\top},
\quad k=1,2,
\]
gives
\begin{align}
W_x\Sigma_x
\left(\bar{E}_1^{-1}\bar{A}_1\right)^{\top}
+
A_2^{\top}W_y\Sigma_y\bar{E}_1^{-\top}
&=
-C_1^{\top}
\left(\bar{E}_1^{-1}\bar{B}_1\right)^{\top},
\label{eq:proj-right1-inv}
\\
W_x\Sigma_x
+
E_2^{\top}W_y\Sigma_y\bar{E}_1^{-\top}
&=
-C_2^{\top}
\left(\bar{E}_1^{-1}\bar{B}_2\right)^{\top}.
\label{eq:proj-right2-inv}
\end{align}
From \eqref{eq:proj-right2-inv},
\[
W_x\Sigma_x
=
-C_2^{\top}
\left(\bar{E}_1^{-1}\bar{B}_2\right)^{\top}
-
E_2^{\top}W_y\Sigma_y\bar{E}_1^{-\top}.
\]
Substituting this into \eqref{eq:proj-right1-inv} yields
\begin{equation}\label{eq:W-main}
A_2^{\top}W_y\Sigma_y\bar{E}_1^{-\top}
-
E_2^{\top}W_y\Sigma_y\bar{E}_1^{-\top}
\left(\bar{E}_1^{-1}\bar{A}_1\right)^{\top}
+
C_1^{\top}
\left(\bar{E}_1^{-1}\bar{B}_1\right)^{\top}
-
C_2^{\top}
\left(\bar{E}_1^{-1}\bar{B}_2\right)^{\top}
\left(\bar{E}_1^{-1}\bar{A}_1\right)^{\top}
=
0.
\end{equation}

Now multiply the Sylvester equation for \(W_{22}\) on the right by
\(\left(\bar{E}_1^{-1}\bar{A}_1\right)^{\top}\). This gives
\[
A_2^{\top}W_{22}
\left(\bar{E}_1^{-1}\bar{A}_1\right)^{\top}
-
E_2^{\top}W_{22}
\left(\bar{E}_1^{-1}\bar{A}_1\right)^{\top}
\left(\bar{E}_1^{-1}\bar{A}_1\right)^{\top}
-
C_2^{\top}
\left(\bar{E}_1^{-1}\bar{B}_2\right)^{\top}
\left(\bar{E}_1^{-1}\bar{A}_1\right)^{\top}
=
0.
\]
Adding this to \eqref{eq:W21} yields
\begin{align*}
& A_2^{\top}
\left(
W_{21}
+
W_{22}\left(\bar{E}_1^{-1}\bar{A}_1\right)^{\top}
\right)
-
E_2^{\top}
\left(
W_{21}
+
W_{22}\left(\bar{E}_1^{-1}\bar{A}_1\right)^{\top}
\right)
\left(\bar{E}_1^{-1}\bar{A}_1\right)^{\top}
\\
&\qquad
+
C_1^{\top}
\left(\bar{E}_1^{-1}\bar{B}_1\right)^{\top}
-
C_2^{\top}
\left(\bar{E}_1^{-1}\bar{B}_2\right)^{\top}
\left(\bar{E}_1^{-1}\bar{A}_1\right)^{\top}
=
0.
\end{align*}
Thus
\[
W_{21}
+
W_{22}\left(\bar{E}_1^{-1}\bar{A}_1\right)^{\top}
\]
satisfies the same Sylvester equation \eqref{eq:W-main} as
\[
W_y\Sigma_y\bar{E}_1^{-\top}.
\]

The equations \eqref{eq:W-main} and \eqref{eq:W21} have the same left-hand Sylvester operator
\[
Z
\mapsto
A_2^{\top}Z
-
E_2^{\top}Z
\left(\bar{E}_1^{-1}\bar{A}_1\right)^{\top}.
\]
Because the solution \(W_{21}\) of \eqref{eq:W21} is unique, the corresponding homogeneous equation has only the zero solution; otherwise adding a nonzero homogeneous solution to \(W_{21}\) would give a second solution of \eqref{eq:W21}. Hence \eqref{eq:W-main} has at most one solution. Therefore
\[
W_y\Sigma_y\bar{E}_1^{-\top}
=
W_{21}
+
W_{22}\left(\bar{E}_1^{-1}\bar{A}_1\right)^{\top}.
\]
Multiplying on the right by \(\bar{E}_1^{\top}\) gives
\[
W_y\Sigma_y
=
W_{21}\bar{E}_1^{\top}
+
W_{22}
\left(\bar{E}_1^{-1}\bar{A}_1\right)^{\top}
\bar{E}_1^{\top}.
\]
Since
\[
\left(\bar{E}_1^{-1}\bar{A}_1\right)^{\top}\bar{E}_1^{\top}
=
\bar{A}_1^{\top},
\]
we obtain
\[
W_y\Sigma_y
=
W_{21}\bar{E}_1^{\top}
+
W_{22}\bar{A}_1^{\top}.
\]
Multiplying on the right by \(\Sigma_y^{-1}\) proves
\[
W_y
=
\left(W_{21}\bar{E}_1^{\top}+W_{22}\bar{A}_1^{\top}\right)\Sigma_y^{-1},
\]
which is \eqref{eq:Wy}.

Finally, use \eqref{eq:proj-right2-inv} again:
\[
W_x\Sigma_x
=
-C_2^{\top}
\left(\bar{E}_1^{-1}\bar{B}_2\right)^{\top}
-
E_2^{\top}W_y\Sigma_y\bar{E}_1^{-\top}.
\]
Substitute the already proved identity
\[
W_y\Sigma_y\bar{E}_1^{-\top}
=
W_{21}
+
W_{22}\left(\bar{E}_1^{-1}\bar{A}_1\right)^{\top}
\]
into this expression:
\[
W_x\Sigma_x
=
-C_2^{\top}
\left(\bar{E}_1^{-1}\bar{B}_2\right)^{\top}
-
E_2^{\top}W_{21}
-
E_2^{\top}W_{22}
\left(\bar{E}_1^{-1}\bar{A}_1\right)^{\top}.
\]
From the Sylvester equation for \(W_{22}\),
\[
A_2^{\top}W_{22}
-
E_2^{\top}W_{22}
\left(\bar{E}_1^{-1}\bar{A}_1\right)^{\top}
-
C_2^{\top}
\left(\bar{E}_1^{-1}\bar{B}_2\right)^{\top}
=
0,
\]
and hence
\[
-C_2^{\top}
\left(\bar{E}_1^{-1}\bar{B}_2\right)^{\top}
-
E_2^{\top}W_{22}
\left(\bar{E}_1^{-1}\bar{A}_1\right)^{\top}
=
-A_2^{\top}W_{22}.
\]
Therefore
\[
W_x\Sigma_x
=
-\left(
E_2^{\top}W_{21}
+
A_2^{\top}W_{22}
\right).
\]
Multiplying on the right by \(\Sigma_x^{-1}\) proves
\[
W_x
=
-\left(
E_2^{\top}W_{21}
+
A_2^{\top}W_{22}
\right)\Sigma_x^{-1},
\]
which is \eqref{eq:Wx}.

Thus all four relations \eqref{eq:Vx}--\eqref{eq:Wy} hold.
\end{proof}

\end{document}